\documentclass[letterpaper,11pt,twoside]{article}

\usepackage[top=1in,bottom=1in,left=1in,right=1in,includehead]{geometry}

\usepackage{amsmath,amsthm,amsfonts,amssymb,upref}

\usepackage[sc]{mathpazo}
\usepackage[T1]{fontenc}
\usepackage{mathrsfs} % Ralph Smith’s Formal Script Font

\usepackage{graphicx,bbm,tikz}
\usetikzlibrary{shapes}

\usepackage{enumitem}
\setenumerate{label=(\alph*)}

\usepackage{authblk}

\newtheorem{theorem}{Theorem}[section]
\newtheorem*{theorem*}{Theorem}

\newtheorem*{lemma*}{Lemma}

\newtheorem{proposition}[theorem]{Proposition}
\newtheorem*{proposition*}{Proposition}

\theoremstyle{definition}

\newtheorem*{definition*}{Definition}
\newtheorem{example}[theorem]{Example}
\newtheorem*{example*}{Example}

\newtheorem*{observation*}{Observation}

\newtheorem*{claim*}{Claim}
\newtheorem{remark}[theorem]{Remark}
\newtheorem*{remark*}{Remark}

\title{Enumerative Theory for the Tsetlin Library}
\makeatletter
\let\sparetitle\@title
\makeatother
\author{Sourav Chatterjee\thanks{souravc@stanford.edu}}
\author{Persi Diaconis\thanks{diaconis@math.stanford.edu}}
\author{Gene B. Kim\thanks{genebkim@stanford.edu}}
\affil{Departments of Mathematics and Statistics}
\affil{Stanford University}

\usepackage{fancyhdr}
\begin{document}

\maketitle

\begin{center}
    In memory of Georgia Benkart
\end{center}

\begin{abstract}
	The Tsetlin library is a well-studied Markov chain on the symmetric group $S_n$. It has stationary distribution $\pi(\sigma)$ the Luce model, a nonuniform distribution on $S_n$, which appears in psychology, horse race betting, and tournament poker. Simple enumerative questions, such as ``what is the distribution of the top $k$ cards?'' or ``what is the distribution of the bottom $k$ cards?'' are long open. We settle these questions and draw attention to a host of parallel questions on the extension to the chambers of a hyperplane arrangement.
\end{abstract}

\section{Introduction}

Let $\theta_1, \theta_2, \dots, \theta_n$ be positive real numbers. The \textbf{Luce model} $\pi(\sigma)$ is a probability distribution on the symmetric group $S_n$ driven by these weights. In words, ``put $n$ balls with weights $\theta_1, \theta_2, \dots, \theta_n$ into an urn. Each time, withdraw a ball from the urn (sampling without replacement) with probability proportional to its weight (relative to the remaining balls).'' Thus, if $w_n = \theta_1 + \cdots + \theta_n$,
\begin{equation}\label{luce}
	\pi(\sigma) = \frac{\theta_{\sigma(1)}}{w_n} \frac{\theta_{\sigma(2)}}{w_n - \theta_{\sigma(1)}} \frac{\theta_{\sigma(3)}}{w_n - \theta_{\sigma(1)} - \theta_{\sigma(2)}} \cdots \frac{\theta_{\sigma(n)}}{\theta_{\sigma(n)}}.
\end{equation}
In Section \ref{backgroundsec}, a host of applied problems are shown to give rise to the Luce model. These include the celebrated Tsetlin library, a Markov chain on $S_n$, described as: ``at each step, choose a card labeled $i$ with probability proportional to $\theta_i$ and move it to the top'' -- $\pi(\sigma)$ is the stationary distribution of this Markov chain.

It is natural to ask basic enumerative questions: pick $\sigma$ from $\pi(\sigma)$.
\begin{itemize}
	\item What is the distribution of the number of fixed points, cycles, longest cycle, and order of $\sigma$?
	\item What is the distribution of the length of the longest increasing subsequence of $\sigma$?
\end{itemize}

Most of these questions are open at this time. The main results below settle
\begin{itemize}
	\item What is the distribution of the top $k$ cards $\sigma(1), \dots, \sigma(k)$? (Section \ref{topsec})
	\item What is the distribution of the bottom $k$ cards $\sigma(n-k+1), \dots, \sigma(n)$? (Section \ref{bottomsec})
\end{itemize}

Weighted sampling without replacement from a finite population is a standard topic. This may be accessed from the Wikipedia entries of ``Horvitz--Thompson estimator'' \cite{HorTho} and ``concomitant order statistics.''

Coming closer to combinatorics, two papers by Rosen \cite{Rosen} treat the coupon collector's problem and other coverage problems.

The recent paper by Ben-Hamou, Peres and Salez \cite{BPS} couples sampling with and without replacement so that tail and concentration bounds, derived for partial sums when sampling with replacement, are seen to apply ``as is'' to sampling without replacement.

A final item; throughout, we have assumed that the weights $\theta_i$ are fixed and known. It is also natural to consider random weights. For a full development, see \cite{PitTran}. 

Section \ref{hyperplanesec} develops the connections of the Tsetlin library to the Bidigare--Hanlon--Rockmore (BHR) walk on the chambers of a real hyperplane arrangement. Understanding the stationary distributions of these Markov chains is almost completely open.

Section \ref{backgroundsec} begins with a review of enumerative group theory. These questions make sense for continuous groups. Georgia Benkart made fundamental contributions here through her work on decomposing tensor products.

\section{Background}\label{backgroundsec}

\subsection{Enumerative Group Theory}

Let $G$ be a finite group. A classical question is ``pick $g \in G$ at random. What does it look like?'' For example, if $G = S_n$,
\begin{itemize}
	\item What is the distribution of $F(g)$, the number of fixed points of $g$?
	\item How many cycles are typical?
	\item What is the expected length of the longest cycle?
	\item What about the length of the longest increasing subsequence of $g$?
	\item What about the descent structure of $g$?
	\item How many inversions are typical?
\end{itemize}
All of these questions have classical answers (references below).

For $G = GL_n(q)$, parallel questions involve the conjugacy class structure of a random $g \in G$. For a splendid development (for finite groups of Lie type), see Fulman \cite{Ful}, which has full references to the results above. The recent survey of Diaconis and Simper \cite{DiaSim} brings this up to date. It focuses on enumeration by double cosets $H \setminus G / K$.

The questions above make sense for continuous groups, where they become ``random matrix theory.'' For example, when $G = O_n$ (the real orthogonal group), one may study the eigenvalues of $g \in G$ under Haar measure by studying the powers of traces
\[
    \int_{O_n} (\text{Tr}(g))^k \, \mathrm{d}g.
\]
Patently this asks for the number of times the trivial representation appears in the $k$th tensor power of the usual $n$-dimensional representation of $O_n$. See \cite{Dia} for details.

Georgia Benkart did extensive work on decomposing tensor powers of representations of classical (and more general) groups. She worked on this with many students and coauthors. Her monograph with Britten and Lemire \cite{BenkBritLemi} is a convenient reference. Most of this work can be translated into probabilistic limit theorems. We started to do this with Georgia during MSRI 2018, but got sidetracked into doing a parallel problem working over fields of prime characteristic in joint work with Benkart--Diaconis--Liebeck--Tiep \cite{BDLT}.

Most all of the above is enumeration under the uniform distribution. A recent trend in enumerative (probabilistic) group theory is enumeration under natural \emph{non}-uniform distributions. For example, on $S_n$,
\begin{itemize}
    \item The Ewens measure $\pi_\theta(\sigma) = Z^{-1}(\theta) \theta^{C(\sigma)}$. Here, $\theta$ is a fixed positive real number, $C(\sigma)$ is the number of cycles of $\sigma$, and $Z^{-1}(\theta)$ is a simple normalizing constant. The Ewens measure originated in biology, but has blossomed into a large set of applications. See Crane \cite{Crane}.
    \item The Mallows measure $\pi_\theta(\sigma) = Z^{-1}(\theta) \theta^{I(\sigma)}$, where $I(\sigma)$ is the number of inversions of $\theta$. This was originally studied for taste testing experiments but has again had a huge development.
    \item More generally, if $G$ is a finite group and $S \subseteq G$ is a symmetric generating set, let $\ell(g)$ be the length function and define $P_\theta(g) = Z^{-1}(\theta) \theta^{\ell(g)}$. Ewens and Mallows models are special cases with $G = S_n$ and $S = \left\{ \text{all transpositions} \right\}$ and $S = \left\{ \text{all adjacent transpositions} \right\}$.
\end{itemize}
Most of the questions studied above under the uniform distribution have been fully worked out under Ewens and Mallows measures. See the survey by Diaconis and Simper \cite{DiaSim} for pointers to a large literature.

The above can be amplified to "permutons" \cite{HKMRS} and "theons" \cite{CorRaz}. It shows that enumeration under non-uniform distributions is an emerging and lively subject. We turn next to the main subject of the present paper.

\subsection{The Luce model}

This section gives several applications where the Luce model appears.

\subsubsection{Psychology}

In psychophysics experiments, a panel of subjects are asked to rank things, such as:
\begin{itemize}
    \item Here are seven shades of red; rank them in order of brightness.
    \item Here are five tones; rank them from high to low.
    \item The same type of task occurs in taste-testing experiments. Rank these five brands of chocolate chip cookies (or wines, etc.) in order of preference.
\end{itemize}
This generates a collection of rankings (permutations) and one tries to draw conclusions.

Patently, rankings vary stochastically; if the same person is asked the same question at a later time, we expect the answers to vary slightly.

Duncan Luce introduce the model \eqref{luce} via the simple idea that each item has a true weight (say, $\theta_i$) and the model \eqref{luce} induces natural variability (which can then be compared with observed data).

Indeed, he did more, crafting a simple set of axioms for pairwise comparison and showing that any consistent ranking distribution has to follow \eqref{luce} for some choice of $\theta_i$. This story is well and clearly told in \cite{Luce1} and \cite{Luce2}.

We would be remiss in not pointing to the widespread dissatisfaction over the ``independence of irrelevant alternatives'' axiom in Luce's derivation. The long Wikipedia article on ``irrelevance of alternatives'' chronicles experiments and theory disputing this, not only for Luce but in Arrow's paradox and several related developments. Amos Tversky's ``elimination by aspects (EBA)'' model is a well-liked alternative.

\subsubsection{Exponential formulation}

Luce's work followed fifty years of effort to model such rankings. Early work of Thurstone and Spearman postulated ``true weights'' $\theta_1, \dots, \theta_n$ for the ordered values and supposed people perceived $\theta_i + \varepsilon_i, 1 \leq i \leq n$ with $\varepsilon_i$ independent normal $\mathcal{N}(0,\sigma^2)$. They then reported the ordering of these perturbed values.

Yellott \cite{Yellott} noticed that if in fact the $\varepsilon_i$ had an extreme value distribution, with distribution function $e^{-e^{-x/r}}, -\infty < x < \infty$, then the associated Thurstonian ranking model is exactly the Luce model!

It is elementary, that if the random variable $Y$ has an exponential distribution ($P(Y > x) = e^{-x}$), then $\log Y$ has an extreme value distribution. This gives the following theorem (used in Section \ref{bottomsec}):

\begin{theorem}
    For $1 \leq i \leq n$, let $X_i$ be independent exponential random variables on $[0,\infty)$ with density
    \[
        \theta_i e^{-x \theta_i}
    \]
    (so $X_i = Y_i/\theta_i$ with $Y_i$ the standard exponential). Then, the chance of the event
    \[
        X_1 < X_2 < \cdots < X_n
    \]
    is (with $w_n = \theta_1 + \cdots + \theta_n$)
    \[
        \frac{\theta_1}{w_n} \cdot \frac{\theta_2}{w_n - \theta_1} \cdot \frac{\theta_3}{w_n - \theta_1 - \theta_2} \cdots \frac{\theta_n}{\theta_n}.
    \]
\end{theorem}

\begin{proof}
    Consider the event $X_1 < X_2 < \cdots < X_n$. The chance of this is
\begin{align*}
	&\theta_1 \cdots \theta_n \int_{x_1 = 0}^\infty \int_{x_2 = x_1}^\infty \cdots \int_{x_n = x_{n-1}}^\infty \exp\left\{ -\sum_{i=1}^n x_i \theta_i \right\} \, dx_1 \cdots dx_n \\
	&\quad = \frac{\theta_1 \cdots \theta_n}{\theta_n} \int_{x_1 = 0}^\infty \cdots \int_{x_{n-1} = x_{n-2}}^\infty \exp\left\{ -\sum_{i=1}^{n-2} x_i \theta_i - x_{n-1}\left( \theta_{n-1} + \theta_n \right) \right\} \, dx_1 \cdots dx_{n-1} \\
	&\quad = \frac{\theta_1 \cdots \theta_n}{\theta_n \left( \theta_n + \theta_{n-1} \right)} \int_{x_1 = 0}^\infty \cdots \int_{x_{n-2} = x_{n-3}}^\infty \exp\left\{ -\sum_{i=1}^{n-3} x_i \theta_i - x_{n-2} \left( \theta_{n-2} + \theta_{n-1} + \theta_n \right) \right\} \, dx_1 \cdots dx_{n-2} \\
	&\quad = \frac{\theta_1 \cdots \theta_n}{\theta_n (\theta_n + \theta_{n-1}) (\theta_n + \theta_{n-1} + \theta_{n-2}) \cdots (\theta_n + \cdots + \theta_1)},
\end{align*}
which is indeed equal to
\[
	\frac{\theta_1}{w_n} \cdot \frac{\theta_2}{w_n - \theta_1} \cdot \frac{\theta_3}{w_n - \theta_1 - \theta_2} \cdots \frac{\theta_n}{\theta_n}.
\]
Thus, the order statistics follow the Luce model \eqref{luce}.
\end{proof}

For an application of the exponential representation to survey sampling, see Gordon \cite{Gordon}.

\subsubsection{Tsetlin library}

The great algebraist Tsetlin was forced to work in a library science institute. While there, he postulated (and solved) the following problem:

Consider $n$ library books arrnged in order $1, 2, \dots, n$. Suppose book $i$ has popularity $\theta_i$. During the day, patrons come and pick up book labeled $i$ with probability $\theta_i/w_n$ and after perusing, they replace it at the left end of the row.

This is a Markov chain on $S_n$ and Tsetlin \cite{Tsetlin} showed that it has \eqref{luce} as its stationary distribution.

The same model has been repeatedly rediscovered; in computer science, the books are discs in deep storage. When a disc is called for, it is replaced on the front of the queue to cut down on future search costs. See Dobrow and Fill \cite{DobFill}.

The model (and its stationary distribution) appear in genetics as the GEM (Griffiths--Engen--McCloskey)  distribution \cite{Donnelly}.

Over the years, a host of properties of the Tsetlin chain have been derived. For example, Phatarfod \cite{Phatarfod} found a simple formula for the eigenvalues and Diaconis \cite{Dia2} found sharp rates of convergence to stationarity (including a cutoff) for a wide class of weights. See further Nestoridi \cite{Nestoridi}. All of this is now subsumed under "hyperplane walks"; see Section \ref{hyperplanesec}.

\textbf{A monotonicity property.} Suppose, without essential loss, that $\theta_1 \geq \theta_2 \geq \cdots \geq \theta_n > 0$. Then,
\begin{itemize}
    \item The largest $\pi(\sigma)$ is for $\sigma = \text{id}$.
    \item The smallest $\pi(\sigma)$ is for $\sigma = (n, n-1, \dots, 1)$.
    \item More generally, $\pi(\sigma)$ is monotone decreasing in the weak Bruhat order on permutations.
\end{itemize}
To explain, the weak Bruhat order is a partial order on $S_n$ with cover relations $\sigma \preceq \sigma'$ if $\sigma$ can be reached from $\sigma'$ by a single adjacent transposition of the $i$th and $(i+1)$th symbols when $\sigma'(i) < \sigma'(i+1)$. Thus, when $n = 3$,
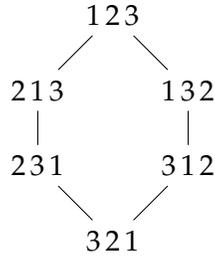
\begin{center}
    \begin{tikzpicture}
        \node (321) at (0,0) {$3 \, 2 \, 1$};
        \node (231) at (-1,1) {$2 \, 3 \, 1$};
        \node (312) at (1,1) {$3 \, 1 \, 2$};
        \node (213) at (-1,2) {$2 \, 1 \, 3$};
        \node (132) at (1,2) {$1 \, 3 \, 2$};
        \node (123) at (0,3) {$1 \, 2 \, 3$};
					
	   \draw (123) -- (213) -- (231) -- (321) -- (312) -- (132) -- (123);
    \end{tikzpicture}
\end{center}

\begin{proposition}
    For $\theta_1 \geq \theta_2 \geq \cdots \geq \theta_n$, $\pi(\sigma)$ is monotone decreasing in the weak Bruhat order.
\end{proposition}

\begin{proof}
    The formulas for $\pi(\sigma)$ and $\pi(\sigma')$ only differ in one term. If $\sigma_i > \sigma_{i+1}$ and these terms were transposed from $\sigma'$, then
    \[
        \frac{\pi(\sigma)}{\pi(\sigma')} = \frac{1 - \theta_{\sigma_1} - \theta_{\sigma_2} - \cdots - \theta_{\sigma_{i+1}}}{1 - \theta_{\sigma_1} - \theta_{\sigma_2} - \cdots - \theta_{\sigma_i}} < 1.
    \]
\end{proof}
\textbf{Irrelevance of alternatives.} The following property characterizes the Luce measure \eqref{luce}: let $\left\{ i_1, i_2, \dots, i_k \right\}$ be a subset of $\left\{ 1, 2, \dots, n \right\}$. Fix $\theta_1, \dots, \theta_n$ and pick $\sigma$ from $\pi(\sigma)$. The distribution of cards labeled $\left\{ i_1, \dots, i_k \right\}$ follows the Luce model with parameters $\theta_{i_1}, \dots, \theta_{i_k}$. For example, the chance that $i$ is above $j$ in $\sigma$ is ${\displaystyle \frac{\theta_i}{\theta_i + \theta_j}}$. This is easy to see from the exponential representation.

\subsubsection{Order statistics and a natural choice of weights}\label{sukhatmesec}

Many questions in probability and mathematical statistics can be reduced to the study of the order statistics of uniform random variables on $[0,1]$ by using the simple fact that, if $X$ is a real random variable with continuous distribution function $F(x)$ (so $P(X \leq x) = F(x)$), then $Y = F(X)$ is uniformly distributed on $[0,1]$. This implies that standard goodness of fit tests (e.g., Kolmogorov--Smirnov) have distributions that are universal under the null hypothesis (they do not depend on $F$). If $Y$ is uniform on $[0,1]$, then $-\log Y$ is standard exponential as above, so order statistics of independent exponentials are a mainstream object of study. A marvelous introduction to this set of ideas is in Chapter 3 of \cite{Feller} with Ronald Pyke's articles on spacings \cite{Pyke} providing deeper results.

With this background, let $Y_1, Y_2, \dots, Y_n$ be independent standard exponentials on $(0,\infty)$. Denote the order statistics by $Y_{(1)} \leq Y_{(2)} \leq \cdots \leq Y_{(n)}$. The following property is easy to prove \cite{Feller}.

\begin{theorem}
    With above notation,
    \[
        Y_{(1)}, Y_{(2)} - Y_{(1)}, Y_{(3)} - Y_{(2)}, \dots, Y_{(n)} - Y_{(n-1)}
    \]
    are independent exponential random variables with distributions
    \[
        Y_{(1)} \sim E_1/n, \qquad Y_{(2)} - Y_{(1)} \sim E_2/(n-1), \qquad \dots \qquad Y_{(n)} - Y_{(n-1)} \sim E_n,
    \]
    where $E_1, \dots, E_n$ are independent standard exponentials (density $e^{-x}$ on $(0,\infty)$).
\end{theorem}

It follows from our Luce calculations that the chance that the smallest spacing is $Y_{(1)}$ is $\displaystyle \frac{n}{\binom{n+1}{2}} = \frac{2}{n+1}$, and that the smallest spacing is $Y_{(2)} - Y_{(1)}$ is $\displaystyle \frac{n-1}{\binom{n+1}{2}}$, and so on. Specifically, $Y_{(j+1)} - Y_{(j)}$ has probability $\displaystyle \frac{n-j}{\binom{n+1}{2}}$ of being smallest, and $Y_{(n)} - Y_{(n-1)}$ has chance $\displaystyle \frac{1}{\binom{n+1}{2}}$ of being smallest.

The whole permutation is given by the Luce model \eqref{luce} with $\theta_i = n-i+1$. This classical fact is due to Sukhatme \cite{Sukhatme}. We will call these \textbf{Sukhatme weights} in the following discussion.

\subsubsection{Application to poker and the ICM (iterated card model)}

In tournament poker (e.g., the World Series of Poker), suppose there are $n$ players at the final table with player $i$ having $\theta_i$ dollars. It is current practice among top players to assume that the order of the players, as they are eliminated, follows the Luce model (with the player having the largest $\theta_i$ least likely to be eliminated; thus most likely to win all the money), and so on. This is called the ICM (iterated card model) and is used as a basis for splitting the total capital and for calculating chances as the game progresses. For careful details and references, see Diaconis--Ethier \cite{DiaEth}, which disputes the model.

\subsubsection{Applications to horse racing}

In horse racing, players can bet on a horse to win, place (come in second), or show (come in third). The ``crowd'' does a good job of determining the chances of each of the $n$ horses running to come in first. Call the amount bet on horse $i$ just before closing, $\theta_i$. However, the crowd does a poor job of judging the chance of a horse showing. Often, there is sufficient disparity between the crowd's bet and the true odds that money can be made (perhaps one race in four). This is despite the track's rake being 17\% of the total. A group of successful bettors uses the $\theta_i$'s and the Luce model to evaluate the chance of placing. For details, see Hausch, Lo and Ziemba \cite{HLZ} or Harville \cite{Harville}.

With this list of applications, we trust we have sufficient motivation to ask ``what does the distribution (1), $\pi(\sigma)$, look like?''

\section{The top $k$ cards}\label{topsec}

Throughout this section, without loss of generality, assume $\theta_1 + \cdots + \theta_n = 1$. For ${\boldsymbol\theta}$ and $k$ fixed, let
\begin{equation}
    P(\sigma_1 \, \sigma_2 \, \cdots \, \sigma_k) = \frac{\theta_{\sigma_1} \theta_{\sigma_2} \cdots \theta_{\sigma_k}}{(1-\theta_{\sigma_1})(1-\theta_{\sigma_1}-\theta_{\sigma_2}) \cdots (1 - \theta_{\sigma_1} - \cdots - \theta_{\sigma_{k-1}})}
\end{equation}
denote the measure induced on the top $k$ cards by the Luce measure. It is cumbersome to compute, e.g.,
\[
    P(\sigma_2) = \theta_{\sigma_2} \sum_{i \neq \sigma_2} \frac{\theta_i}{1-\theta_i}.
\]
On the other hand, the Luce measure is just sampling from an urn without replacement. If $\left\{ \theta_i \right\}$ are ``not too wild'' and $k$ is small, then sampling with or without replacement should be ``about the same.'' This is made precise in two metrics.

Let $Q(\sigma_1 \, \sigma_2 \, \cdots \, \sigma_k)$ be the product measure
\begin{equation}
    Q(\sigma_1 \, \sigma_2 \, \cdots \, \sigma_k) = \theta_{\sigma_1} \theta_{\sigma_2} \cdots \theta_{\sigma_k},
\end{equation}
where $\sigma_1,\ldots, \sigma_k$ need not be distinct. Both $P$ and $Q$ depend on $\left\{ \theta_i \right\}$ and $k$, but this is suppressed below. Define
\[
    d_\infty(P, Q) = \max_\sigma \left( 1 - \frac{Q(\sigma_1 \, \cdots \, \sigma_k)}{P(\sigma_1 \, \cdots \, \sigma_k)} \right)
\]
and
\[
    \| P - Q \|_{TV} = \frac{1}{2} \sum_\sigma \left\lvert P(\sigma_1 \, \cdots \, \sigma_k) - Q(\sigma_1 \, \cdots \, \sigma_k) \right\rvert,
\]
where, in both formulas, $\sigma_1,\ldots,\sigma_k$ are not necessarily distinct, and $P(\sigma_1\, \cdots\,\sigma_k) =0$ if they are not distinct. Clearly, $\| P - Q \|_{TV} \leq d_\infty(P, Q)$.

\begin{theorem}\label{distthm1}
    For $\theta_1 + \cdots + \theta_n = 1$, $\theta_i \leq \frac{1}{2}$ for all $1 \leq i \leq n$,
    \[
        d_\infty(P,Q) \leq 1 - \exp\left\{ -2\left( (k-1)\theta_{(1)} + (k-2)\theta_{(2)} + \cdots + \theta_{(k-1)} \right) \right\}.
    \]
    Here, $\theta_{(1)} \geq \theta_{(2)} \geq \cdots \geq \theta_{(n)}$ are the ordered values.
\end{theorem}

\begin{theorem}\label{distthm2}
    As $n \to \infty$, suppose $\displaystyle \binom{k}{2} \sum_{i=1}^n \theta_i^2 \to \lambda$. Then,
    \[
        \| P - Q \|_{TV} \sim 1 - e^{-\lambda}.
    \]
\end{theorem}

In Theorem 3.2, $\left\{ \theta_i \right\}$ form a triangular array, but again, this is suppressed in the notation. The remarks below point to non-asymptotic versions.

\begin{proof}[Proof of Theorem \ref{distthm1}]
    From the definitions,
    \[
        d_\infty(P,Q) = \max_\sigma \left( 1 - (1-\theta_{\sigma_1})(1 - \theta_{\sigma_1} - \theta_{\sigma_2}) \cdots (1 - \theta_{\sigma_1} - \cdots - \theta_{\sigma_{k-1}} ) \right),
    \]
    where the maximum is over all $\sigma_1,\ldots, \sigma_k$ distinct (because, if they are not distinct, then we have that ${\displaystyle 1- \frac{Q(\sigma_1\, \cdots\, \sigma_k)}{P(\sigma_1\, \cdots\, \sigma_k)} = -\infty}$, which does not contribute to the maximum). 
    Use $-2x \leq \log(1-x) \leq -x$ for $0 \leq x \leq \frac{1}{2}$. Since all $\theta_i \leq \frac{1}{2}$,
    \[
        d_\infty(P,Q) \leq \max_\sigma 1 - \exp\left\{ -2 \left( \theta_{\sigma_1} + (\theta_{\sigma_1} + \theta_{\sigma_2}) + \cdots + (\theta_{\sigma_1} + \cdots + \theta_{\sigma_{k-1}} ) \right) \right\}.
    \]
    The right-hand side is maximized for $\sigma_1, \dots, \sigma_{k-1}$ with the largest weights.
\end{proof}

\begin{proof}[Proof of Theorem \ref{distthm2}]
    A prepatory observation is useful:
    \[
        \| P - Q \|_{TV} = \sum_{\sigma: P(\sigma) \geq Q(\sigma)} (P(\sigma) - Q(\sigma)) = 1 - P_Q(\text{$\sigma_1, \dots, \sigma_k$ are distinct}).
    \]
    This is just the chance that there are two or more balls in the same box if $k$ balls are dropped independently into $n$ boxes, the chance of box $i$ being $\theta_i$. This non-uniform version of the classical birthday problem has been well-studied. If $X_{ij}$ is 1 or 0 as balls $i,j$ are dropped into the same box and
    \[
        W = \sum_{1 \leq i < j \leq k} X_{ij},
    \]
    $\displaystyle E(W) = \binom{k}{2} \sum_{i=1}^n \theta_i^2$. Under the condition $E(W) \to \lambda$, $W$ is known to have a limiting Poisson($\lambda$) distribution and $P_Q(W = 0) \sim e^{-\lambda}$. See Chatterjee--Diaconis--Meckes \cite{CDM} or Barbour--Holst--Janson \cite{BHJ} for further details and more quantitative bounds.
\end{proof}

\begin{example}
    Consider the Sukhatme weights from Section \ref{sukhatmesec}:
    \[
        \theta_i = \frac{n+1-i}{\binom{n+1}{2}}.
    \]
    The exponent for the right-hand side of Theorem \ref{distthm1} is
    \[
        \frac{2}{\binom{n+1}{2}} \left\{ (k-1)n + (k-2)(n-1) + \cdots (n-k+2) \right\}.
    \]
    Simple asymptotics show that for $k = c\sqrt{n}$, $c > 0$, this is
    \[
        4c^2 + \mathcal{O}\left( \frac{1}{\sqrt{n}} \right).
    \]
    So,
    \[
        d_\infty(P,Q) \leq 1 - e^{-4c^2 + \mathcal{O}\left( \frac{1}{\sqrt{n}} \right)},
    \]
    and $k \ll \sqrt{n}$ suffices for product measure. To be a useful approximation to the first $k$-coordinates of the Luce measure with $k = c\sqrt{n}$,
    \[
        \binom{k}{2} \sum_{i=1}^n \theta_i^2 \sim \frac{c^2}{3} = \lambda
    \]
    giving a similar approximation in total variation.
\end{example}

\begin{example}
    The bound in Theorem \ref{distthm1} is useful when
    \[
        (k-1) \theta_{(1)} + \cdots + \theta_{(k-1)}
    \]
    is small. To see that this condition is needed, take $\theta_1 = \frac{1}{2}$, $\theta_i = \frac{1}{2(n-1)}$ for $2 \leq i \leq n$. For $k = 2$,
    \[
        P(1 \, 2) = \frac{\theta_1 \theta_2}{1 - \theta_1}, \qquad Q(1 \, 2) = \theta_1 \theta_2,
    \]
    and so, $d_\infty(P,Q) = 1 - (1 - \theta_1) = \frac{1}{2}$. This does not tend to zero when $n$ is large. The two-sided bounds for $\log(1-x)$ show Theorem \ref{distthm1} is sharp in this sense for general $k$.
\end{example}

\begin{example}
    As discussed above, if the infinity distance tends to zero, then total variation tends to zero. Here is a choice of weights $\theta_i$ so that the total variation convergence holds for the joint distribution of the first $k$ coordinates of the Luce model to i.i.d.\ is close, but not in infinity distance.

    Fix $k, 1 \leq k \leq n$ and let $\theta_i = k^{-7/4}$ for $i \leq k$ and
    \[
        \theta_i = \frac{1 - k^{-3/4}}{n-k}
    \]
    for $i > k$ so that $\theta_1 \geq \theta_2 \geq \cdots \geq \theta_n > 0$ and $\displaystyle \sum_{i=1}^n \theta_i = 1$. Note that
    \[
        \binom{k}{2} \sum_{i=1}^n \theta_i^2 \leq k^2 \sum_{i=1}^k \theta_i^2 + k^2 \sum_{i=k+1}^n \theta_i^2 \leq k^{-1/2} + \frac{k^2}{n-k}
    \]
    while
    \[
        \sum_{i=1}^{k-1} (k-i)\theta_i = k^{-7/4} \sum_{i=1}^{k-1} (k-i) = \frac{1}{2} \left( k^{-3/4} (k-1) \right).
    \]
    Thus, if $1 \ll k \ll \sqrt{n}$, $\displaystyle \sum_{i=1}^{k-1} (k-i) \theta_i$ is large, but $\displaystyle \binom{k}{2} \sum_{i=1}^n \theta_i^2$ is small. Moreover, if $k = \lambda \sqrt{n}$, then $\displaystyle \binom{k}{2} \sum_{i=1}^n \theta_i^2 \to \frac{\lambda^2}{2}$, but $\displaystyle \sum_{i=1}^{k-1} (k-i)\theta_i \to \infty$.
\end{example}

\begin{remark}
    \begin{enumerate}
        \item Our proof of Theorem \ref{distthm2} used the Poisson approximation for the non-uniform version of the birthday problem. There are other possible limits which can be used to bound $\| P - Q \|_{TV}$. See \cite{CamPit}.
        \item It is easy to see that
        \[
            Q_\theta(\sigma_1, \dots, \sigma_k \text{ distinct}) = k! e_k(\theta_1, \theta_2, \dots, \theta_n),
        \]
        where $e_k$ is the $k$th elementary symmetric function. From here, Muirhead's theorem shows $\| P - Q \|_{TV}$ is a Schur-concave function of $\theta_1, \dots, \theta_n$, smallest when $\theta_i = \frac{1}{n}$.
    \end{enumerate}
\end{remark}

\section{The bottom $k$ cards}\label{bottomsec}

\subsection{Introduction}

For naturally occurring weights, the bottom $k$ cards behave very differently from the top $k$ cards. To illustrate by example, consider the Sukhatme weights of Section \ref{sukhatmesec}:
\[
    \theta_i = \frac{i}{\binom{n+1}{2}}, \qquad 1 \leq i \leq n.
\]
The results of Section \ref{topsec} show that, for large $n$,
\[
    P\left( \frac{\sigma_1}{n} \leq x \right) \sim 2 \int_0^x (1-y) \, \mathrm{d}y.
\]
That is, $\sigma_1/n$ has a limiting $\beta(1,2)$ distribution.

Using Theorems \ref{distthm1} and \ref{distthm2}, the same holds for $\sigma_i/n$ for fixed $i \ll \sqrt{n}$. Of course, large numbers have higher probabilities, but all values in $\left\{ 1, 2, \dots, n \right\}$ occur.

In contrast, consider the value of bottom card $\sigma_n$. Intuitively, this should be small since the high numbers have higher weights. We were surprised to find
\[
    P(\sigma_n = 1) \sim 0.516\dots
\]
In fact, we computed, using a result that follows, that
\begin{center}
	\begin{tabular}{|c|l|}\hline
		$\ell$ & $P(\text{$\ell$ is last})$ \\ \hline
		1 & 0.516094 \\
		2 & 0.213212 \\
		3 & 0.107310 \\
		4 & 0.0597505 \\
		5 & 0.0354888 \\
		6 & 0.0220716 \\
		7 & 0.0142167 \\
		8 & 0.00941619 \\
		9 & 0.00638121 \\
		10 & 0.00440862 \\ \hline
	\end{tabular}
\end{center}
The section below sets up its own notation from first principles.

\subsection{Main result}

Let $\mathbb{N}$ denote the set of positive integers and let $\mathbb{N}^\mathbb{N}$ be the set of all maps from $\mathbb{N}$ into $\mathbb{N}$. Consider the topology of pointwise convergence on $\mathbb{N}^\mathbb{N}$. This topology is naturally metrizable with a complete separable metric, and so we can talk about convergence of probability measures on this space.

Now suppose that for each $n$, $\sigma_n$ is a random element of the symmetric group $S_n$. We can extend $\sigma_n$ to a random element of $\mathbb{N}^\mathbb{N}$, by defining $\sigma_n(i)=i$ for $i>n$.

\begin{proposition}\label{convprop}
    Let $\sigma_n$ be as above. Then $\sigma_n$ converges in law as $n\to\infty$ if and only if for each $k$, the random vector $(\sigma_n(1),\ldots,\sigma_n(k))$ converges in law as $n\to\infty$.
\end{proposition}

\begin{proof}
    Since the coordinate maps on $\mathbb{N}^\mathbb{N}$ are continuous in the topology of pointwise convergence, one direction is clear.
    
    For the other direction, suppose that for each $k$, $(\sigma_n(1),\dots,\sigma_n(k))$ converges in law as $n\to\infty$. Notice that for any sequence of positive integers $a_1,a_2,\dots$, the set
    \begin{equation}\label{kdef}
        \left\{ f\in \mathbb{N}^\mathbb{N} : \text{$f(i) \leq a_i$ for all $i$} \right\}
    \end{equation}
    is a compact subset of $\mathbb{N}^\mathbb{N}$, since any infinite sequence in this set has a convergent subsequence by a diagonal argument. Take any $\varepsilon >0$. By the given condition, $\sigma_n(i)$ converges in law as $n\to\infty$ for each $i$. In particular, $\{\sigma_n(i)\}_{n\ge 1}$ is a tight family, and so there is some number $a_i$ such that for each $n$,
    \[
        P(\sigma_n(i) > a_i) \leq 2^{-i} \varepsilon.
    \]
    Therefore if $K$ denotes the set defined in \eqref{kdef} above, then for each $n$, 
    \[
        P(\sigma_n \in K) \ge 1-\sum_{i=1}^\infty P(\sigma_n(i) > a_i)\ge 1-\sum_{i=1}^\infty 2^{-i} \varepsilon = 1-\varepsilon.
    \]
    This proves that $\{\sigma_n\}_{n\ge 1}$ is a tight family of random variables on $\mathbb{N}^\mathbb{N}$. Therefore the proof will be complete if we can show that any probability measure on $\mathbb{N}^\mathbb{N}$ is determined by its finite dimensional distributions. But this is an easy consequence of Dynkin's $\pi$-$\lambda$ theorem.
\end{proof}

The above proposition implies, for instance, that if $\sigma_n$ is a uniform random element of $S_n$, then $\sigma_n$ does not converge in law on $\mathbb{N}^\mathbb{N}$, because $\sigma_n(1)$ does not converge in law.

Let $0 < \theta_1 \leq \theta_2 \leq \cdots $ be a non-decreasing infinite sequence of positive real numbers. For each $n$, consider the Luce model on $S_n$ with parameters $\theta_1,\ldots,\theta_n$. Let $\sigma_n$ be the {\it reverse} of a random permutation drawn from this model. That is, $\sigma_n(1)$ is the last ball that was drawn and $\sigma_n(n)$ is the first.
%For each $n$, let $\sigma_n$ be a random element of $S_n$ defined as follows. First, pick an index between $1$ and $n$ with probability proportional to $\theta_i$. Having chosen the index, next choose another index between $1$ and $n$, but excluding the index already chosen, and again with probability proportional to $\theta_i$. Keep on doing this until all numbers between $1$ and $n$ have been chosen. Let $\sigma_n(n)$ be the first index that was chosen, let $\sigma_n(n-1)$ be the second index, and so on.
As we know from prior discussions, an equivalent definition is the following. Let $X_1,X_2,\ldots$ be an infinite sequence of independent random variables, where $X_i$ has exponential distribution with mean $1/\theta_i$. Then $\sigma_n\in S_n$ is the permutation such that $X_{\sigma_n(1)} > X_{\sigma_n(2)} >\cdots> X_{\sigma_n(n)}$.  

%A standard argument shows that this $\sigma_n$ has the same distribution as the one defined in the previous paragraph.

\begin{theorem}\label{convthm}
    Let $\sigma_n$ be as above. For each $x \geq 0$, let
    \[
        f(x) := \sum_{i=1}^\infty e^{-\theta_i x},
    \]
    where we allow $f(x)$ to be $\infty$ if the sum diverges. Let 
    \[
        x_0 := \inf \left\{x : f(x) < \infty \right\},
    \] 
    with the convention that the infimum of the empty set is $\infty$. Then $\sigma_n$ converges in law as $n\to\infty$ if and only if $x_0<\infty$ and $f(x_0)=\infty$.  Moreover, if this condition holds, then the limiting finite dimensional probability mass functions are given by the following formula: For any $k$ and any distinct positive integers $a_1,\ldots, a_k$, 
    \begin{align*}
        \lim_{n\to\infty} P(\sigma_n(1)=a_1,\dots, \sigma_n(k)=a_k) = \int_{x_1>x_2>\cdots > x_k>0} \prod_{j=1}^k(\theta_{a_j} e^{-\theta_{a_j} x_j}) \prod_{i\notin  \{a_1,\ldots,a_k\}} (1-e^{-\theta_i x_{k}}) \, \mathrm{d}x_1\cdots \mathrm{d}x_k.
    \end{align*}
\end{theorem}

Before proving the theorem, let us work out some simple examples. Suppose that $\theta_i = i$ for each $i$. This corresponds to the Luce model with the Sukhatme weights. Then clearly $f(x)<\infty$ for all $x>0$, and hence $x_0=0$. Also, clearly, $f(0)=\infty$. Therefore in this case $\sigma_n$ converges in law as $n\to\infty$. Moreover, by the formula displayed above,
\begin{align*}
    \lim_{n\to\infty} P(\sigma_n(1)=1) = \int_0^\infty e^{-x}\prod_{j=2}^\infty (1-e^{-jx}) dx = \int_0^1 \prod_{j=2}^\infty (1-y^j) dy. 
\end{align*}
On the other hand, for the case of uniform random permutations, $\theta_i=1$ for all $i$. In this case, $f(x)=\infty$ for all $x$, and hence $x_0=\infty$. Thus, the theorem implies that $\sigma_n$ does not converge in law (which we know already).

Next, suppose that $\theta_i = \beta \log (i+1)$ for some $\beta >0$. Here $f(x)<\infty$ for $x>1/\beta$ and $f(x)=\infty$ for $x\leq 1/\beta$. Thus, $x_0=1/\beta$ and $f(x_0)=\infty$, and so by the theorem, $\sigma_n$ converges in law. 

Strangely, $\sigma_n$ {\it does not} converge  in law if $\theta_i = \log (i+1) + 2\log \log (i+1)$. To see this, note that in this case,
\[
    f(x) = \sum_{i=1}^\infty \frac{1}{(i+1)^x (\log (i+1))^{2x}}.
\]
Thus, $f(x)<\infty$ for $x>1$ and $f(x)=\infty$ for $x<1$, showing that $x_0=1$. But 
\[
    f(x_0) = \sum_{i=1}^\infty \frac{1}{(i+1)(\log(i+1))^2} <\infty,
\]
which violates the second criterion required for convergence. This shows that we cannot determine  convergence purely by inspecting the rate of growth of $\theta_i$. The criterion is more subtle than that.

What happens if the tightness criterion does not hold? In this case, the formula for the limit of $P(\sigma_n(1)=a_1,\dots, \sigma_n(k)=a_k)$ remains valid, but it may not represent a probability mass function, i.e., the sum over all $a_1,\ldots, a_k$ may be strictly less than $1$.

\begin{proof}[Proof of Theorem \ref{convthm}]
    Take any $k\ge 1$ and distinct positive integers $a_1,\dots, a_k$. Take $n \geq \max_{1 \leq i \leq k} a_i$. Let $E_n$ be the event $\{\sigma_n(1)=a_1,\dots, \sigma_n(k)=a_k\}$. Then
    \begin{align*}
        &P(E_n) = P(X_{a_1} > X_{a_2}>\cdots > X_{a_k} > X_i \ \forall i\in[n]\setminus \{a_1,\ldots,a_k\})\\
        &= \int_{x_1>x_2>\cdots > x_k>0} \prod_{j=1}^k(\theta_{a_j} e^{-\theta_{a_j} x_j}) \prod_{i\in [n]\setminus \{a_1,\ldots,a_k\}} (1-e^{-\theta_i x_{k}}) \, \mathrm{d}x_1\cdots \mathrm{d}x_k.
    \end{align*}
    By the dominated convergence theorem, this gives
    \begin{align*}
        \lim_{n\to\infty} P(E_n) = \int_{x_1>x_2>\cdots > x_k>0} \prod_{j=1}^k(\theta_{a_j} e^{-\theta_{a_j} x_j}) \prod_{i\notin  \{a_1,\dots,a_k\}} (1-e^{-\theta_i x_{k}}) \, \mathrm{d}x_1\cdots \mathrm{d}x_k.
    \end{align*}
    Thus, we have shown that for any $k$ and distinct positive integers $a_1,\ldots,a_k$, $\lim_{n\to\infty} P(\sigma_n(1)=a_1,\ldots,\sigma_n(k)=a_k)$ exists, and also found the desired formula for the limit. However, we have not shown convergence in law because we have not established tightness. (This is not surprising, because we did not use any properties of the $\theta_i$'s yet.) From what we have done until now, it follows that $(\sigma_n(1), \dots, \sigma_n(k))$ converges in law as $n\to\infty$ if and only if it is a tight family. But this holds if and only if $\{\sigma_n(i)\}_{n\geq 1}$ is a tight family for every $i$. We will now complete the proof of the theorem by showing that $\{\sigma_n(i)\}_{n\geq 1}$ is a tight family for every $i$ if and only if $x_0 < \infty$ and $f(x_0)=\infty$. 

    First, suppose that $\{\sigma_n(1)\}_{n\ge 1}$ is a tight family. Then there is some $a$ such that 
    \[
    \lim_{n\to\infty} P(\sigma_n(1)=a)>0.
    \]
    From the above calculation, we know that
    \begin{align*}
        \lim_{n\to\infty}P(\sigma_n(1)=a) &= \int_0^\infty \theta_a e^{-\theta_a x} \prod_{i\neq a} (1-e^{-\theta_i x}) \, \mathrm{d}x.
    \end{align*}
    If this is nonzero, then there is at least one $x>0$ for which 
    \[
        \prod_{i=1}^\infty (1-e^{-\theta_i x}) >0.
    \]
    But this implies that
    \[
        f(x) = \sum_{i=1}^\infty e^{-\theta_i x} <\infty. 
    \]
    Thus, $x_0<\infty$. Next, we show that $f(x_0)=\infty$. Suppose not. Then $x_0>0$, since $f(0)=\infty$. Fix a positive integer $a$. For each $n\geq a$, and let $A_n$ be the event $\left\{ \sigma_n(1) \leq a \right\}$. Let $F_n$ be the event $\left\{ \max_{i\leq n} X_i \leq x_0 \right\}$. Take any $x \in (0,x_0)$ and let $G_n$ be the event $\left\{ \max_{i\leq n} X_i \leq x \right\}$. Then
    \begin{align*}
        P(A_n) &\leq P(A_n \cap (F_n\setminus G_n)) + P((F_n\setminus G_n)^c)\\
        &= P(A_n\cap (F_n\setminus G_n)) + P(F_n^c \cup G_n)\\
        &\leq P(A_n\cap (F_n\setminus G_n))  + P(G_n) + P(F_n^c).
    \end{align*}
    If the event $A_n\cap(F_n\setminus G_n)$ happens, then $\max_{i\leq n} X_i$ belongs to the interval $(x,x_0]$, and one of $X_1,\dots, X_a$ is the maximum among $X_1,\dots,X_n$. Thus, in particular, one of $X_1,\ldots,X_a$ is in $(x,x_0]$. Plugging this into the above inequality, we get
    \begin{align*}
        P(A_n) &\leq \sum_{i=1}^a (e^{-\theta_i x} - e^{-\theta_i x_0}) + \prod_{i=1}^n (1-e^{-\theta_i x}) + 1 - \prod_{i=1}^n(1-e^{-\theta_i x_0}).
    \end{align*}
    Since $f(x)=\infty$, we have $\displaystyle \prod_{i=1}^\infty (1-e^{-\theta_i x}) = 0$. Thus, taking $n \to \infty$ on both sides, we get
    \begin{align*}
        \lim_{n\to\infty} P(A_n) &\leq \sum_{i=1}^a (e^{-\theta_i x} - e^{-\theta_i x_0}) +   1 - \prod_{i=1}^\infty(1-e^{-\theta_i x_0}).
    \end{align*}
    Now notice that the definition of $A_n$ does not involve $x$. So we can take $x\nearrow x_0$ on the right, which makes the first term vanish and leaves the rest as it is. Thus,
    \[
        \lim_{n\to\infty} P(A_n) \leq 1 - \prod_{i=1}^\infty(1-e^{-\theta_i x_0}).
    \]
    But the assumed finiteness of $f(x_0)$ implies that the product on the right is strictly positive. Thus, we get an upper bound on $\displaystyle \lim_{n\to\infty} P(A_n)$ which is less than $1$. But observe that this upper bound does not depend on $a$. This contradicts the tightness of $\sigma_n(1)$, thereby completing the proof of one direction of the theorem.

    Next, suppose that $x_0<\infty$ and $f(x_0)=\infty$. 
    We consider two cases. First, suppose that $x_0=0$. Then $f(x)<\infty$ for each $x>0$. But 
    \begin{align}\label{fdef2}
        f(x) = \sum_{i=1}^\infty P(X_i > x). 
    \end{align}
    Therefore by the Borel--Cantelli lemma, $X_i \to 0$ almost surely as $i \to \infty$. Now take any $i$ and integers $n$ and $a$ bigger than $i$. Then the event $\sigma_n(i) \geq a$ implies that 
    \[
        \max_{j \geq a} X_j > \min \left\{X_1, \dots ,X_i \right\},
    \] 
    because otherwise the $i$th largest value among $(X_j)_{j=1}^n$ cannot be one of $(X_j)_{j\geq a}$. Thus, 
    \[
        P(\sigma_n(i) \geq a) \leq P\left(\max_{j\ge a} X_j > \min \left\{ X_1, \dots, X_i \right\} \right).
    \]
    But the right side is a function of only $a$ (and not $n$), and tends to zero as $a\to \infty$ because $X_j \to 0$ almost surely as $j \to \infty$. This proves tightness of $\left\{ \sigma_n(i) \right\}_{n \geq 1}$ when $x_0=0$.

    Next, consider the case $x_0>0$. For convenience, let us define the partial sums
    \[
        f_n(x) := \sum_{i=1}^n e^{-\theta_i x}, \qquad g_n(x) := \prod_{i=1}^n \left(1-e^{-\theta_i x} \right). 
    \]
    Take $i$, $n$ and $a$ as before. Let $x$ be a real number bigger than $x_0$, to be chosen later. The event $\sigma_n(i) \geq a$ implies that at least one of the following two events must happen: (a) There are less than $i$ elements of $(X_j)_{j=1}^n$ that are bigger than $x$, or (b) $X_j>x$ for some $j\geq a$. This gives
    \begin{align*}
        P(\sigma_n(i)\geq a) &\leq \sum_{A\subseteq[n], \lvert A \rvert < i} \left( \prod_{j\in A} e^{-\theta_j x} \right) \left(\prod_{j\in [n]\setminus A} \left( 1 - e^{-\theta_j x} \right) \right) + \sum_{j\geq a} e^{-\theta_j x}.
    \end{align*}
    Now note that for any $A\subseteq [n]$ with $|A|< i$, 
    \begin{align*}
        \prod_{j\in [n]\setminus A} \left( 1-e^{-\theta_j x} \right)  &\leq \frac{g_n(x)}{\displaystyle \prod_{j\in A} \left( 1-e^{-\theta_j x} \right)} \leq \frac{g_n(x)}{\left( 1-e^{-\theta_1 x_0} \right)^{i-1}}. 
    \end{align*}
    Therefore 
    \begin{align*}
        \sum_{A\subseteq[n], \ |A|< i} \left( \prod_{j\in A} e^{-\theta_j x} \right) \left( \prod_{j\in [n]\setminus A} \left( 1-e^{-\theta_j x} \right) \right) & \leq \frac{g_n(x)}{(1-e^{-\theta_1 x_0})^{i-1}}\sum_{A\subseteq [n], \ |A|< i} \left(\prod_{j\in A} e^{-\theta_j x} \right)\\
        &\leq \frac{g_n(x)(1+f_n(x)+f_n(x)^2+\cdots+ f_n(x)^{i-1})}{(1-e^{-\theta_1 x_0})^{i-1}}.
    \end{align*}
    By the inequality $1-x \leq e^{-x}$, we have $g_n(x) \leq e^{-f_n(x)}$. Thus,
    \begin{align*}
        P(\sigma_n(i) \geq a) &\leq \frac{e^{-f_n(x)}(1+f_n(x)+f_n(x)^2+\cdots+ f_n(x)^{i-1})}{(1-e^{-\theta_1 x_0})^{i-1}} + \sum_{j\ge a} e^{-\theta_j x}.
    \end{align*}
    Let $m$ be the largest integer such that $\theta_m \leq 1/(x-x_0)$. Suppose that $n \geq m$. Then 
    \begin{align*}
        f_n(x) &\geq \sum_{j=1}^m e^{-\theta_j x} = \sum_{j=1}^m e^{-\theta_j x_0 - \theta_j(x-x_0)} \geq e^{-1}\sum_{j=1}^m e^{-\theta_j x_0}. 
    \end{align*}
    But $m\to\infty$ as $x \searrow x_0$, and $f(x_0)=\infty$ by assumption. Thus, the above inequality shows that given any $L > 0$, we can first choose $x$ sufficiently close to $x_0$, and then choose $n_0$ sufficiently large, such that for all $n \geq n_0$, $f_n(x) \geq L$. Now take any $\varepsilon >0$ and find $L$ so large that for all $y \geq L$,
    \[
        \frac{e^{-y}(1+y+y^2+\cdots+ y^{i-1})}{(1-e^{-\theta_1 x_0})^{i-1}} < \frac{\varepsilon}{2}. 
    \]
    Choose $x$ and then $n_0$ as in the previous paragraph corresponding to this $L$. Then find $a$ so large that 
    \[
        \sum_{j \geq a} e^{-\theta_j x}< \frac{\varepsilon}{2}, 
    \]
    which exists since $f(x) <\infty$. For this choice of $a$, the above steps show that $P(\sigma_n(i) \geq a) \leq \varepsilon$ for all $n \geq n_0$. This proves tightness of $\left\{ \sigma_n(i) \right\}_{n \geq 1}$ when $x_0>0$, completing the proof of the theorem. 
\end{proof}

\section{A vast generalization -- hyperplane walks}\label{hyperplanesec}

\subsection{Introduction}

The Tsetlin library has seen vast generalizations in the past twenty years. In this section, we explain walks on the chambers of a hyperplane arrangement due to Bidigare--Hanlon--Rockmore \cite{BHR} and Brown--Diaconis \cite{Brown2}. The Tsetlin library is a (very) special case of the braid arrangement. These Markov chains have a fairly complete theory (simple forms for the eigenvalues and good rates of convergence to stationarity). But the description of the stationary distribution, the analog of the Luce model, is indirect, involving a weighted sampling without replacement scheme. Thus the problem
\begin{center}
    \emph{What does the stationary distribution of hyperplane walks look like?}
\end{center}
Section \ref{hypersec} sets things up and states the main theorems (with examples). The few cases where something is known are reported in Section \ref{stationarysec}. The final section points to semigroup walks where parallel problems remain open. The main point of this section is to cast Sections \ref{backgroundsec}--\ref{bottomsec} above as contributions to a general problem.

\subsection{Hyperplane walks}\label{hypersec}

We work in $\mathbb{R}^d$. Let $\mathcal{A} = \left\{ H_1, H_2, \dots, H_k \right\}$ be a finite collection of affine hyperplanes (translates of codimension one subspaces). These divide $\mathbb{R}^d$ into
\begin{itemize}
    \item chambers (points not on any $H_i$). Let $\mathcal{C}$ be the chambers.
    \item faces (points on some $H_i$ and on one side or another of others). Let $\mathcal{F}$ be the faces.
\end{itemize}

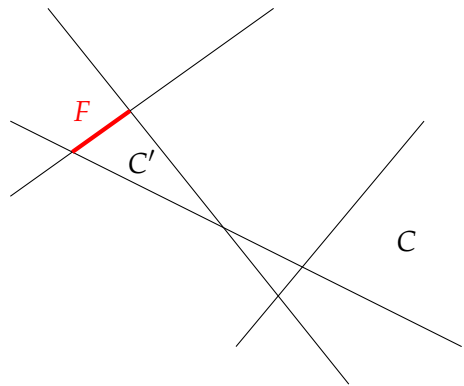
\begin{figure}[h!]
    \centering
     \begin{tikzpicture}
        \draw (-3,1) -- (3,-2);
        \draw (-2.5,2.5) -- (1.5,-2.5);
        \draw (-3,0) -- (0.5,2.5);
        \draw (0,-2) -- (2.5,1);
        \draw[ultra thick,red] ({-37/17},{10/17}) -- ({-31/22},{25/22});
        \node[above left] at ({(-37/17 - 31/22)/2}, {(10/17 + 25/22)/2}) {\color{red}$F$};
        \node[below] at (-1.25,0.75) {$C'$};
        \node[right] at (2,-0.6) {$C$};
    \end{tikzpicture}
    \caption{Four lines in $\mathbb{R}^2$. There are 10 chambers and 30 faces (chambers, points of intersection and the empty face are faces).}
\end{figure}

A key notion is the projection of a chamber onto a face (Tits projection). For $C \in \mathcal{C}$ and $F \in \mathcal{F}$, $\text{PROJ } C \to F$ is the unique chamber adjacent to $F$ and closest to $C$ (in the sense of crossing the fewest number of $H_i$'s). In the above figure, $\text{PROJ } C \to F = C'$.

With these definitions, we are ready to walk. Choose face weights $\left\{ w_F \right\}_{F \in \mathcal{F}}$ with $w_F \geq 0$ and $\displaystyle \sum_{F \in \mathcal{F}} w_F = 1$. Define a Markov chain $\kappa(C, C')$ on chambers via:
\begin{itemize}
    \item from $C$, choose $F \in \mathcal{F}$ with probability $w_F$ and move to $\text{PROJ } C \to F$.
\end{itemize}
Thus, $\displaystyle \kappa(C,C') = \sum_{\text{PROJ }C \to F = C'} w_F$.

\begin{example}[Boolean arrangements]
    Let $H_i = \left\{ x \in \mathbb{R}^d : x_i = 0 \right\}, 1 \leq i \leq d$ be the usual coordinate hyperplanes. These divide $\mathbb{R}^d$ into $2^d$ chambers (the usual orthants) and $3^d$ faces. A face may be labeled by a vector of length $d$ containing $0, \pm1$ to delineate 0 (on) or on one side or the other of the $i$th hyperplane. Chambers are faces with no zeros. For $\text{PROJ } C \to F = C'$, set the $i$th coordinate of $C'$ to the $i$th coordinate of $F$ if this is $\pm 1$ and leave it as the $i$th coordinate of $C$ if the $i$th coordinate of $F$ is 0.

    Thus, a walk proceeds via: from $C$, pick a subset of coordinates and install $\pm1$ in them as determined by $F$. For example, if
    \[
        w_F = \begin{cases} \displaystyle \frac{1}{2d} & F = (0,\dots,0,\pm1,0,\dots,0), \\ 0 & \text{otherwise,} \end{cases}
    \]
    the walk becomes ``pick a coordinate at random and replace it with $\pm1$ chosen uniformly.'' This is the celebrated Ehrenfest urn model of statistical physics. Dozens of natural specializations of these Boolean walks are spelled out in \cite{Brown2}.
\end{example}

\begin{example}[Braid arrangement]
    Take $H_{ij} = \left\{ x \in \mathbb{R}^d: x_i = x_j \right\}, 1 \leq i < j \leq d$. Now, the chambers are points in $\mathbb{R}^d$ with no equal coordinates. It follows that the relative order is fixed within a chamber, so chambers can be labeled by permutations. The faces are indexed by ``block ordered set partitions'': coordinates within a block are equal and all coordinates in the first block are smaller than the coordinates in the second block, and so on.

    For the projection, suppose the chamber labeled $\pi$ is thought of as a deck of cards in arrangement $\pi$ (with $\pi(i)$ the label of the card at position $i$). Suppose $d = 5$ and the face is $F = 1 \, 3 / 2 / 4 \, 5$. Remove cards labeled 1 and 3 from $\pi$ (keeping them in their same relative order, then remove the card labeled 2 and place it under cards 1, 3. Finally, remove cards labeled 4, 5 and place them at the bottom of the five card deck. This is $\text{PROJ } \pi \to 1 \, 3 / 2 / 4 \, 5$.

    The \textbf{Tsetlin library} arises from the choice
    \[
        w_F = \begin{cases} \theta_i & \text{if $F = i / [n] \setminus i$} \\ 0 & \text{otherwise} \end{cases}.
    \]
    That is the walk on $S_n$ with "choose label $i$ with probability $\theta_i$ and move this card to the top."

    \textbf{Riffle shuffling} arises from
    \[
        w_F = \begin{cases} \displaystyle \frac{1}{2^d} & \text{if $F = S / [n] \setminus S$ for $S \subseteq [n]$} \\ 0 & \text{otherwise} \end{cases}.
    \]
    Another way to say this -- label each of $d$ cards in the current deck with a fair coin flip, remove all cards labeled ``heads'' keeping them in their same relative order, and place them on top. This is exactly ``inverse riffle shuffling,'' the inverse of the Gilbert--Shannon--Reeds model studied by Bayer--Diaconis \cite{BayDia}.
\end{example}
There are hundreds of other hyperplane arrangements where the chambers are labeled by natural combinatorial objects, and there are choices of face weights so that the walk is a natural object ot study. Indeed, any finite reflection group leads to a hyperplane arrangement with $H_\mathbf{v}$ being the hyperplane orthogonal to the vector $\mathbf{v}$ determining the reflection. Any finite graph leads to a "graphical arrangement." For a wonderful exposition, see Stanley \cite{Stanley}.

As said, the Markov chains $\kappa(C,C')$ admit a complete theory with known eigenvalues and rates of convergence. We will not spell this out here; see \cite{Brown2}, but turn to the main object of interest -- the stationary distribution.

Let $\mathcal{A}$ be a general arrangement with chosen face weights $\left\{ w_F \right\}_{F \in \mathcal{F}}$ and $\kappa(C,C')$ the associated Markov chain on $C$, the chambers of the arrangement. $\pi(C) \geq 0$ and $\sum_{C} \pi(C) = 1$ is stationary for $\kappa$ if $\sum_C \pi(C) \kappa(C,C') = \pi(C')$ -- thus $\pi$ can be thought of as a left eigenvector with eigenvalue 1. When does a unique such $\pi$ exist?

\begin{theorem}[Brown--Diaconis]
    Call $\left\{ w_F \right\}$ \textbf{separating} if they are not all supported in the same hyperplane (for $H \in \mathcal{A}$, there exists $H' \in \mathcal{A}$ and $w_F > 0$ for $F \subset H'$). Then $\kappa$ has a unique stationary distribution $\pi(C)$ if and only if $\left\{ w_F \right\}$ are separating.
\end{theorem}

This $\pi$ is the analog of the Luce model and becomes the Luce model for the braid arrangement as above. The following result gives a ``weighted sampling without replacement characterization'' of $\pi(C)$.

\begin{theorem}[Brown--Diaconis]
    Suppose $\left\{ w_F \right\}$ are separating. The following algorithm generates a pick from $\pi(C)$:
    \begin{itemize}
        \item place all $\left\{ w_F \right\}$ in an urn.
        \item draw them out, without replacement, with probability proportional to size (relative to what is left).
        \item say this results in the ordered list $F_1, F_2, \dots, F_{\lvert \mathcal{F} \rvert}$.
        \item from any starting chamber $C$ (the choice does not matter), project on $F_{\lvert \mathcal{F} \rvert}$, then on $F_{\lvert \mathcal{F} \rvert - 1}$, and so on until $F_1$. The resulting chamber is exactly distributed as $\pi(C)$.
    \end{itemize}
\end{theorem}

Of course, for the Tsetlin library, this is just the Luce measure on permutations. The following subsection delineates the few examples where something can be said about $\pi$.

\subsection{Understanding $\pi$}\label{stationarysec}

Suppose a group of orthogonal transformations acts transitively on $\mathcal{A}$ preserving $\kappa(C,C')$. Then, $\pi(C)$ is uniform over $\mathcal{C}$ (supposing separability). Examples include riffle shuffles, the Ehrenfest urn, and ``random to top'' (the Tsetlin library with $\theta_i = \frac{1}{n}, 1 \leq i \leq n$). For more on this, see \cite{Nestoridi}.

Simple features of $\pi$ can sometimes be calculated directly. See Pike \cite{Pike} and its references.

Aside from the present paper, the only other examples that have been carefully studied are in the following graph coloring problems.

\subsubsection{Graph coloring}

Let $G$ be a connected and undirected simple graph. Let $\mathcal{X}$ be the set of 2-colorings (say by $\pm$) of the vertex set of $G$. Define a Markov chain on $\mathcal{X}$ by
\begin{itemize}
    \item from $x \in \mathcal{X}$
    \item pick an edge $e \in G$ uniformly at random
    \item change the two endpoints of $e$ in $x$ to be
    \begin{tikzpicture}
        \draw (0,0) -- (1,0);
        \node[above] at (0,0) {+};
        \node[above] at (1,0) {+};
        \node[below] at (0.5,0) {$e$};
        \fill (0,0)  circle[radius=2pt];
        \fill (1,0)  circle[radius=2pt];
    \end{tikzpicture}
    or
    \begin{tikzpicture}
        \draw (0,0) -- (1,0);
        \node[above] at (0,0) {$-$};
        \node[above] at (1,0) {$-$};
        \node[below] at (0.5,0) {$e$};
        \fill (0,0)  circle[radius=2pt];
        \fill (1,0)  circle[radius=2pt];
    \end{tikzpicture}
    with probability $\frac{1}{2}$.
\end{itemize}
Thus ``neighbors are inspired to match, at random times.'' This is a close cousin of standard particle systems such as the voter model. All the theory works. The process is a hyperplane walk for the Boolean arrangement of dimension $D$, where $D$ denotes the number of edges in the graph $G$. All eigenvalues and rates of convergence are easily available.

The only thing open is
\begin{center}
    ``what can be said about the stationary distribution?''
\end{center}
To understand the question, suppose the graph is an $n$-point path
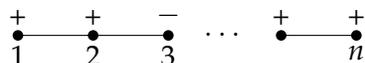
\begin{center}
    \begin{tikzpicture}
        \draw (0,0) -- (2,0) (3.5,0) -- (4.5,0);
        \fill (0,0)  circle[radius=2pt];
        \fill (1,0)  circle[radius=2pt];
        \fill (2,0)  circle[radius=2pt];
        \fill (3.5,0)  circle[radius=2pt];
        \fill (4.5,0)  circle[radius=2pt];
        \node[above] at (0,0) {+};
        \node[above] at (1,0) {+};
        \node[above] at (2,0) {$-$};
        \node[above] at (3.5,0) {+};
        \node[above] at (4.5,0) {+};
        \node[below] at (0,0) {1};
        \node[below] at (1,0) {2};
        \node[below] at (2,0) {3};
        %\node[below] at (3,0) {$n-1$};
        \node[below] at (4.5,0) {$n$};
        \node at (2.75,0) {$\cdots$};
    \end{tikzpicture}
\end{center}
The distribution $\pi$ is far from uniform. All $+$ or all $-$ have chance $\frac{1}{2}$ of staying, but $+-+-\cdots$ is impossible. Of course, $\pi(x)$ is invariant under switching $+$ and $-$. It is easy to show that, under $\pi$, the $\pi$ process is a 1-dependent point process (see \cite{BDF}). This means various central limit theorems are available.

How much more likely is ``all $+$'' than ``many alternations''? This problem was carefully studied in a difficult paper by Chung and Graham \cite{ChungGraham} (see also \cite{BCG}). They show, under $\pi$, all $+$ (or all $-$) have chance of order $C/2^n$, but many alternations has chance of order $C'/n!$. Very nice systems of recursive differential equations appear.

The point is, even in the simplest case, understanding the stationary distribution leads to interesting mathematics. We offer the present paper in this spirit.

\subsection{Semigroups and beyond}

The past ten years have shown yet broader generalization of the Tsetlin library. Kenneth Brown extended it to idempotent semigroups (allowing walks on the chambers of a building) \cite{Brown}. Ben Steinberg, working wiht many coauthors, extended further in the semigroup direction. A convenient reference is the book length treatment \cite{MSS}.

In another direction, a sweeping generalization of much of moderm algebra based on hyperplane and semigroup walks has been developed by Aguiar and Mahajan \cite{AM1,AM2,AM3}. The three large volumes contain hundreds of fresh examples.

In \emph{none} of these developments is the stationary measure understood.

\section*{Acknowledgement}

Thanks to Jim Pitman for useful references.

\end{document}